\documentclass{amsart}
\usepackage{amssymb,a4wide}
\usepackage{amsthm}
\usepackage{graphicx}
\usepackage[T2A]{fontenc}
\usepackage[latin1]{inputenc}

\renewcommand{\i}{{i'}}
\renewcommand{\j}{{j'}}

\newcommand{\R}{\mathbb R}

\newcommand{\Mu}{{\mathcal M}}
\newcommand{\Id}{{\bf 1}}

\newcommand{\weg}[1]{}
\newcommand{\spann}{\mathrm{span}}

\theoremstyle{plain}
\newtheorem{thm}{Theorem}
\newtheorem*{thm*}{Theorem}
\newtheorem{lem}{Lemma}

\newtheorem{cor}{Corollary}
\newtheorem*{con}{Convention}

\theoremstyle{definition}

\theoremstyle{remark}
\newtheorem{rem}{Remark}

\title{The Tanno-Theorem for Kählerian metrics with arbitrary signature}
  \author{A. Fedorova, S. Rosemann}

\thanks{supported by  GK 1523}
\begin{document}

\begin{abstract}
Considering a non-constant smooth solution $f$ of the Tanno equation on a closed, connected Kähler manifold $(M,g,J)$ with positively definite metric $g$, Tanno showed that the manifold can be finitely covered by $(\mathbb{C}P(n),\mbox{const}\cdot g_{FS})$, where $g_{FS}$ denotes the Fubini-Study metric of constant holomorphic sectional curvature equal to $1$. The goal of this paper is to give a proof of Tannos Theorem for Kähler metrics with arbitrary signature.  
\end{abstract}

\maketitle

\section{Introduction}
\label{sec:hpr}

\subsection{Tanno equation and Main result}
Let $(M,g,J)$ be a pseudo-Riemannian Kähler manifold of real dimension $2n$. During the whole paper, we are using tensor notation, for example $T^{i_{1}...i_{p}}_{j_{1}...j_{q},k}$ means covariant differentiation of the tensor $T^{i_{1}...i_{p}}_{j_{1}...j_{q}}$ with respect to the Levi-Civita connection of $g$. If $\omega_{i}$ is a $1$-form on $M$, then $\bar{\omega}_{i}=J^{\alpha}_{i}\omega_{\alpha}$. Moreover, the Kählerian $2$-form is denoted by the symbol $J_{ij}=g_{i\alpha}J^{\alpha}_{j}$.
We are interested in the question, if $(M,g,J)$ allows the existence of a non-constant solution $f$ of the equation
\begin{align}
f_{,ijk}+c(2f_{,k}g_{ij}+f_{,i}g_{jk}+f_{,j}g_{ik}-\bar{f}_{,i}J_{jk}-\bar{f}_{,j}J_{ik})=0\label{eq:tanno}
\end{align}
Originally this equation appeared first in spectral geometry, see \cite{Tanno1978}. Let $(\mathbb{C}P(n),g_{FS})$ be the complex projective space with Fubini-Study metric $g_{FS}$ of constant holomorphic sectional curvature equal to $1$. Then, the eigenfunctions of the Laplacian to the first eigenvalue $-(n+1)$ satisfy \eqref{eq:tanno} with $c=\frac{1}{4}$. 
\begin{rem}
Conversely, contracting \eqref{eq:tanno} with $g^{ij}$ shows that $(\Delta f)_{,k}=-4c(n+1)f_{,k}$, hence $$\Delta(f+C)=-4c(n+1)(f+C)$$ for any solution $f$ of \eqref{eq:tanno} and some constant $C$.
\end{rem}
The goal of this paper is to give a proof of the following Theorem:
\begin{thm}
\label{thm:main}
   Let  $f$ be  a non-constant smooth function on a closed, connected pseudo-Riemannian Kähler manifold
  $(M^{2n},g,J)$ such that the equation (\ref{eq:tanno})
    is  fulfilled for some constant $c$.  Then $c\neq 0$ and $(M^{2n}, 4cg, J)$ can be finitely covered by $(\mathbb{C}P(n), g_{FS}, J_{standard})$.
\end{thm}
\begin{rem}
In particular we obtain that on a closed, connected Kähler manifold $(M,g,J)$ such that $Cg$ is not positively definite for any constant $C\neq 0$, any solution of \eqref{eq:tanno} is necessarily a constant. If we assume in addition lightlike completeness, there is a simple and short proof of this statement. Indeed, let $f$ be a solution of equation \eqref{eq:tanno} and let $\gamma:\mathbb{R}\rightarrow M$ be a lightlike geodesic. Restricting \eqref{eq:tanno} to $\gamma$ yields the ordinary differential equation $f'''(t)=0$, where $f(t)=f(\gamma(t))$. The general solution $f(t)=At^{2}+Bt+C$ must have vanishing constants $A$ and $B$ since $\gamma$ is complete and $f$ is bounded on the closed manifold $M$. We obtain that each solution $f$ of equation \eqref{eq:tanno} is constant along any lightlike geodesic. Since two arbitrary points can be connected by a broken lightlike geodesic, we obtain that all solutions of \eqref{eq:tanno} are constant.
Theorem \ref{thm:main} generalises this result to the case of a closed, connected lightlike incomplete manifold.
\end{rem} 
Theorem \ref{thm:main} was proven in~\cite{Tanno1978} for positively definite $g$ and $c>0$. In this case it is sufficient to require that the manifold is complete. In ~\cite{HiramatuK} it was proven that the equation with positively definite $g$ on a closed manifold can not have non-constant solutions for $c\leq  0$.
\subsection{The Riemannian analogue: Gallot-Tanno equation}
Let us recall that the equation (\ref{eq:tanno}) was introduced  in \cite{Tanno1978} as ``Kählerization'' of 
\begin{align}
f_{,ijk}+c(2f_{,k} \cdot g_{ij}+ f_{,i} g_{j k} + f_{,j} g_{i k})=0\label{eq:GallotTanno}
\end{align}
As its Kählerian analogue \eqref{eq:tanno}, this equation has its origin in spectral geometry: Consider the Laplacian on the sphere $S^{n}$ of constant sectional curvature equal to $1$. Then the eigenfunctions corresponding to the second eigenvalue $-2(n+1)$ satisfy \eqref{eq:GallotTanno} with $c=1$, see \cite{Tanno1978,Gallot1979}. Moreover, this equation appeared independently in many other parts of differential geometry. It is known that a non-constant solution of \eqref{eq:GallotTanno} on a Riemannian manifold $(M,g)$ implies the decomposability of the cone $(M'=\mathbb{R}_{>0}\times M,g'=dr^{2}+r^{2}g)$ (decomposability means that $(M',g')$ locally looks like a product manifold), see \cite{Gallot1979}. In \cite{Mounoud2010,Matveev2010} it was shown that this remains to be true in the pseudo-Riemannian situation for a cone over a closed manifold. Furthermore, equation \eqref{eq:GallotTanno} is related to conformal and projective differential geometry (see \cite{Hiramatu,Tanno1978} and \cite{Matveev2010,Mounoud2010} for references).  
The classical result dealing with equation \eqref{eq:GallotTanno} states that on a complete, connected Riemannian manifold, the existence of a non-constant solution of \eqref{eq:GallotTanno} with $c>0$ implies that the manifold can be covered by the euclidean sphere $S^{n}$. This was proven by Gallot and Tanno independently, see \cite{Tanno1978,Gallot1979}. Recently, this result was generalised to the pseudo-Riemannian situation under the additional assumption that the manifold is closed, see \cite{Mounoud2010,Cortes2009}. For non-closed manifolds, generalisations of the classical result of Gallot and Tanno for metrics with arbitrary signature where discussed in more detail in \cite{Cortes2009}. 
\subsection{Organisation of the paper}
The goal of this paper is to give a proof of Theorem \ref{thm:main}. 
In Section \ref{sec:c=0} we consider the case when the constant $c$ in the equation (\ref{eq:tanno}) is equal to zero. We show that there are no non-constant solutions of  (\ref{eq:tanno}) with $c=0$ on a closed, connected manifold. The relation between non-constant solutions of \eqref{eq:tanno} and holomorph-projective geometry (the Kählerian analogue of projective geometry) will be shown in Section \ref{sec:hpro}. Given a solution of (\ref{eq:tanno}), we can construct a solution of a linear PDE system which appears in the theory of holomorph-projectively equivalent Kähler metrics. Going a step further, in Section \ref{sec:proj} we asign to each solution of \eqref{eq:tanno} a $(1,1)$-tensor on $\widehat{M}:=\mathbb{R}^{2}\times M$. The family of tensors constructed in this way is invariant with respect to the operation of real polynomials on endomorphisms of the tangent bundle $T\widehat{M}$. This fact allows us to find special solutions of \eqref{eq:tanno} such that the corresponding $(1,1)$-tensor acts as a non-trivial projector on $T\widehat{M}$. Using such a solution, the final goal of this Section is to show that the metric $g$ is Riemannian up to multiplication with a constant. The last step in the proof of Theorem \ref{thm:main} is an application of the Theorem proven by Tanno for positive-definite $g$ and will be done in Section \ref{sec:proof}.
\section{The case when $c=0$}
\label{sec:c=0}
Let us now treat the case when the constant $c$ in the equation (\ref{eq:tanno}) is equal to zero. We show 
\begin{thm}
\label{thm:c=0}
Let $(M,g,J)$ be a closed, connected pseudo-Riemannian Kähler manifold and $f$ a solution of $f_{,ijk}=0$. Then $f$ is a constant. 
\end{thm}
\begin{proof}
Let $p$ and $q$ be points of $M$ where $f$ takes its maximum and minimum value respectively. Then $f_{,i}(p)=f_{,i}(q)=0$ and the hessian satisfies $f_{,ij}(p)\leq 0$ and $f_{,ij}(q)\geq 0$. By assumption $f_{,ij}$ is parallel. It follows that $f_{,ij}=0$ implying that $f_{,i}$ is parallel. Since $f_{,i}(p)=$ we obtain that $f_{,i}=0$ on the whole of $M$, hence $f$ is a constant.
\end{proof}

\section{Solutions of (\ref{eq:tanno}) correspond to solutions of Frobenius system}
\label{sec:hpro}
 Let $f$ be a non-constant solution of equation~\eqref{eq:tanno} on a closed, connected manifold $M$. By Theorem \ref{thm:c=0} we obtain $c\ne 0$. The constant $c$ can therefore be included in the metric $g$ without changing the Levi-Civita connection. For simplicity, we denote the new metric again with the symbol $g$, hence $f$ satisfies equation (\ref{eq:tanno}) with $c=1$. 
Consider the symmetric $(0,2)$-tensor $a_{ij}$ and the function $\mu$ defined by 
\begin{gather} 
  a_{ij}:=-f_{,ij}-2fg_{ij}\mbox{ and }\mu:=-2f\label{eq:def}
\end{gather}
Then we obtain that the following linear system of PDE`s is satisfied:
  \begin{gather}
    \label{eq:system}
    \begin{array}{c}
      a_{i j,k}=f_i g_{j k} + f_j g_{i k}-\bar{f}_i J_{jk}-\bar{f}_j J_{ik}\\
      f_{i,j}=\mu g_{i j}-a_{i j},\\
      \mu_{,i}=-2f_i
    \end{array}
  \end{gather} 
Indeed,  covariantly differentiating $a_{ij}=-f_{,ij}-2 f g_{ij}$ and substituting 
\eqref{eq:tanno} yields  
\begin{align}
  a_{ij,k}&=-f_{,ijk}-2 f_{,k} g_{ij}=2 f_{,k} \cdot g_{ij}+ f_{,i} g_{j k} +
    f_{,j} g_{i k} - \bar{f}_{,i}J_{jk}-\bar{f}_{,j}J_{ik} -2 f_{,k}
    g_{ij}\nonumber\\
    &=f_{,i} g_{j k}+f_{,j} g_{i k}-\bar{f}_{,i}J_{jk}-\bar{f}_{,j}J_{ik},\nonumber
\end{align} 
which is the first equation in \eqref{eq:system}. The second and third equations of \eqref{eq:system} are equivalent to the definition \eqref{eq:def}.
\begin{rem}
\label{rem:hpro}
The first equation 
\begin{align}
 a_{i j,k}=f_i g_{j k} + f_j g_{i k}-\bar{f}_i J_{jk}-\bar{f}_j J_{ik}\label{eq:main}
\end{align}
in \eqref{eq:system} is the main equation in holomorph-projective geometry, see for example \cite{Sin2,DomMik1978} or \cite{Mikes} for a survey on this topic. Two Kähler metrics $g$ and $\bar{g}$ are said to be \textit{holomorph-projectively equivalent} if their holomorphically-planar curves coincide. Recall that a curve $\gamma:I\rightarrow M$ is called \textit{holomorphically-planar} with respect to the Kähler metric $g$ if there are functions $\alpha,\beta:I\rightarrow\mathbb{R}$ such that $\nabla_{\dot{\gamma}}\dot{\gamma}=\alpha \dot{\gamma}+\beta J\dot{\gamma}$ is satisfied, see \cite{Otsuki1954,Tashiro1956}. The importance of equation \eqref{eq:main} is due to the fact, that the solutions $(a_{ij},f_{i})$ of \eqref{eq:main}, such that $a_{ij}$ is symmetric and hermitian, are in $1:1$ correspondence with metrics $\bar{g}$ beeing holomorph-projectively equivalent to $g$. 
\end{rem}
\begin{rem}
\label{rem:hpro2}
We want to mention that the $1$-form $f_{i}$ corresponding to a solution $(a_{ij},f_{i})$ of \eqref{eq:main} always is the differential of a function. Indeed, contracting equation \eqref{eq:main} with $g^{ij}$ shows that $f_{i}=\frac{1}{4}a^{\alpha}_{\alpha,i}$, in particular $f_{i,j}=f_{j,i}$.
\end{rem}
Let us note that the construction \eqref{eq:def} is invertible: If there is a solution $(a_{ij},f_{i},\mu)$ of (\ref{eq:system}) we can use the equations in (\ref{eq:system}) to show that $\mu$ always constitutes a solution of equation (\ref{eq:tanno}) with $c=1$. Indeed,
\begin{align}
\mu_{,ijk}&\overset{\mbox{\tiny\eqref{eq:system}}}{=}\nabla_{k}(-2f_{i,j})\overset{\mbox{\tiny\eqref{eq:system}}}{=}\nabla_{k}(-2\mu \tilde{g}_{ij}+2a_{ij})\nonumber\\
&=-2\mu_{,k} \tilde{g}_{ij}+2a_{ij,k}\overset{\mbox{\tiny\eqref{eq:system}}}{=}-2\mu_{,k} \tilde{g}_{ij}+2f_{i}\tilde{g}_{ik}+2f_{j}\tilde{g}_{jk}-2\bar{f}_{i}\tilde{J}_{jk}-2\bar{f}_{j}\tilde{J}_{ik}\nonumber\\
&\overset{\mbox{\tiny\eqref{eq:system}}}{=}-2\mu_{,k} \tilde{g}_{ij}-\mu_{,i}\tilde{g}_{ik}-\mu_{,j}\tilde{g}_{jk}+\bar{\mu}_{,i}\tilde{J}_{jk}+\bar{\mu}_{,j}\tilde{J}_{ik}\nonumber
\end{align}
Hence, given a solution $(a_{ij},f_{i},\mu)$ of \eqref{eq:system} we can define $f:=-\frac{1}{2}\mu$ which gives a solution of \eqref{eq:tanno}. This definition is the inverse construction to \eqref{eq:def} and therefore, the solutions $f$ of \eqref{eq:tanno} and $(a_{ij},f_{i},\mu)$ of \eqref{eq:system} are in linear $1:1$ correspondence.\\
Let us now show the main advantage of working with a system like (\ref{eq:system}):
\begin{lem}
  \label{lem:frobenius}
  Let $(a_{ij},f_i,\mu)$ be a solution of the
  system~(\ref{eq:system}) such that $a_{ij}=0$, $f_i=0$,
  $\mu=0$ at some point $p$ of the connected Kähler manifold
  $(M,g,J)$. Then $a_{ij}\equiv 0$, $f_i\equiv 0$, $\mu\equiv 0$ at all
  points of $M$. 
\end{lem}

\begin{proof}
  The system \eqref{eq:system} is in Frobenius form, i.e., the
  derivatives of the unknowns $a_{ij}, f_i, \mu$ are expressed
  as (linear) functions of the unknowns:
  $$
  \left(\begin{array}{c}a_{ij,k}\\f_{i,j}\\\mu_{,i}\end{array}\right)
  =F\left(\begin{array}{c}a_{ij}\\f_{i}\\\mu\end{array}\right)\nonumber,
  $$
  and all linear systems in the Frobenius form have the property that
  the vanishing of the solution at one point implies the vanishing at all
  points.
\end{proof}

\section{If there is a non-constant solution of \eqref{eq:tanno} with $c=1$, then $g$ is positively definite}
\label{sec:proj}
The goal of this section is to prove the following
\begin{thm}
  \label{pos}
  Let $(M, g, J)$ be a closed, connected Kähler manifold. Suppose $f$ is a non-constant solution of \eqref{eq:tanno} with $c=1$. Then, the metric $g$ is positively definite.
\end{thm}
First we want to show that it is possible to choose one solution of equation (\ref{eq:tanno}) such that the corresponding $(1,1)$-tensor $a^i_j=g^{i\alpha}a_{\alpha j}=-f^{i}_{,j}-2f\delta^{i}_{j}$ has a clear and simple structure of eigenspaces and eigenvectors. Evaluating $a_{ij},f_{,ij}$ and $g_{ij}$ on these eigenspaces will show that $g$ has to be positively definite. 

\subsection{Matrix of the extended system} \label{99}
In order to find the special solution of (\ref{eq:tanno}) mentioned above, we asign to each solution $f$ of (\ref{eq:tanno}) a  $(1,1)$-tensor on the $(2n+2)$-dimensional manifold $\widehat{M}=\R^2\times M$ with coordinates $(\underbrace{x_+,x_-}_{\R^2},\underbrace{x_1,\dots,x_{2n}}_{M})$. For every solution $f$ of equation (\ref{eq:tanno}), let us consider the $(2n+2)\times(2n+2)$-matrix
\begin{gather}
  \label{eq:operator}
  L(f)=\left(
    \begin{array}{cc|ccc}
      \mu&0&f_{1}&\dots&f_{2n}\\
      0&\mu&\bar{f}_{1}&\dots&\bar{f}_{2n}\\
      \hline
      f^{1}&\bar{f}^{1}&&&\\
      \vdots&\vdots&&a^i_j&\\
      f^{2n}&\bar{f}^{2n}&&&
    \end{array}\right)
\end{gather}
where $a_{ij}$ and $\mu$ are defined by the equations \eqref{eq:def}. The matrix $L(f)$ is a well-defined $(1,1)$-tensor field on
$\widehat{M}$ (in the sense that after a change of coordinates on $M$ the components of the matrix $L$ transform
according to tensor rules).

\begin{rem}
  Using \eqref{eq:def} we see that the identity transformation of $T\widehat{M}$ corresponds to the function $f=-\frac{1}{2}$, i.e.
  \begin{gather}
\nonumber    
L\left(-\frac{1}{2}\right)=\left(
      \begin{array}{cc|ccc}
        1&0&0&\dots&0\\
        0&1&0&\dots&0\\
        \hline
        0&0&&&\\
        \vdots&\vdots&&\delta^i_j&\\
        0&0&&&
      \end{array}\right)={\bf 1}
  \end{gather}
\end{rem}
\begin{rem}
  The matrix $L$ contains the information on the function $f$ and its first and second derivatives. By equation \eqref{eq:def} and Lemma \ref{lem:frobenius}, if this matrix is vanishing at some point of $\widehat{M}$ then $f\equiv 0$ on the whole of $M$. In the next section,
  we will see that the matrix formalism does have advantages: we will
  show that the polynomials of the matrix $L$ also correspond to
  certain solutions of equation \eqref{eq:tanno}.
\end{rem}
\begin{rem}  
The constructions which where done here are visually similar to those which can be done for equation \eqref{eq:GallotTanno}. However, in the non-Kählerian situation the corresponding $(1,1)$-tensor on the cone manifold is covariantly constant with respect to the connection induced by the cone metric, see \cite{Gallot1979,Mounoud2010,Matveev2010}. This is not the case for the extended operator \eqref{eq:operator} which poses additional difficulties.
\end{rem}

\subsection{Algebraic properties of $L$}
Obviously, the mapping $f\longmapsto L(f)$ applied to arbitrary smooth functions $f$ on $M$ is a linear injective mapping between the space of smooth functions on $M$ and the space of $(1,1)$-tensors on $\widehat{M}$. For two smooth functions $F$ and $H$ on $M$ let us define a new product 
\begin{align}
F*H:=-2FH-\frac{1}{2}F_{,\alpha}H^{\alpha}_{,}\mbox{ and the $k$-fold potency }F^{*k}:=\underbrace{F*...*F}_{k\mbox{ \tiny times}},\nonumber
\end{align}
The product $*$ is bilinear and commuting but not associative. However, it turns out that the mapping $L$ now preserves the potencies $f^{*k}$ of solutions of \eqref{eq:tanno}:
\begin{lem}
  \label{lem:product}
\begin{enumerate}
  \item Let $f$ be a solution of equation \eqref{eq:tanno} with $c=1$. Then, for every $k\geq 0$ there exists a solution $\tilde{f}$ such that
  $$L^k(f)=L(\tilde {f}), \mbox{ where }L^k = \underbrace{L
    \cdot ... \cdot L}_\text{$k$ times}.$$
\item 
If $f$ is a solution of \eqref{eq:tanno} with $c=1$, then $L^{k}(f)=L(f^{*k})$. In particular, $f^{*k}$ is a solution of \eqref{eq:tanno}. 
\end{enumerate}
\end{lem}
\begin{proof}
  $(1)$ Given two solutions $f$ and $F$ of equation \eqref{eq:tanno}, we denote by $(a_{ij},f_i,\mu)$ and $(A_{ij},F_i,\Mu )$ the corresponding solutions of \eqref{eq:system}, constructed by \eqref{eq:def}. After direct calculation we obtain for the product of the corresponding matrices $L(f)$ and $L(F)$:
 \begin{multline}
 \label{eq:product}
 L(f)\cdot L(F)=\left(
 \begin{array}{cc|ccc}
 \mu \Mu +f_k F^k&f_k \bar{F}^k&\mu
 F_{1}+f_k A^k_1&\dots&\mu F_{2n}+f_k
 A^k_{2n}\\
 \bar{f}_k F^k&\mu \Mu +f_k F^k&\mu
 \bar F_{1}+\bar f_k
 A^k_1&\dots&\mu\bar{F}_{2n}+\bar{f}_k A^k_{2n}\\
 \hline
 \vphantom{\dfrac{1}{2}}\Mu f^{1}+a^1_kF^k&\Mu \bar{f}^{1}+a^1_k\bar{F}^k&&&\\
 \vdots&\vdots&\multicolumn{3}{c}{a^i_k
 A^k_j+f^i F_j+\bar{f}^i\bar{F}_j}\\
 \Mu f^{2n}+a^{2n}_k F^k&
 \Mu \bar{f}^{2n}+a^{2n}_k\bar{F}^k&&&
 \end{array}\right)
 \end{multline}
 Suppose that
 \begin{gather}
 \label{eq:op_eq}
 \mu F_j+f_k A^k_j= \Mu f_j+a^k_j F_k\quad\mbox{and}\quad f^k \bar{F}_k=0
 \end{gather}
 then $L(f)\cdot L(F)$ takes the form \eqref{eq:operator}
 \begin{gather}
 \label{eq:product_sol}
 L(f)\cdot L(F)=\left(
    \begin{array}{cc|ccc}
      \tilde{\mu}&0&\tilde{f}_{1}&\dots&\tilde{f}_{2n}\\
      0&\tilde{\mu}&\bar{\tilde{f}}_{1}&\dots&\bar{\tilde{f}}_{2n}\\
      \hline
      \tilde{f}^{1}&\bar{\tilde{f}}^{1}&&&\\
      \vdots&\vdots&&\tilde{a}^i_j&\\
      \tilde{f}^{2n}&\bar{\tilde{f}}^{2n}&&&
    \end{array}\right)
\end{gather}
 where $\tilde a^{i}_{j}=a^i_k
 A^k_j+f^i F_j+\bar{f}^i\bar{F}_j,\tilde {f}_i=\mu F_i+ A^k_i f_k$ and $\tilde\mu=\mu \Mu +f_k F^k$. Now we show that $\tilde {a}_{ij}$, $\tilde {f}_i$ and $\tilde \mu$
 satisfy~(\ref{eq:system}). In addition, we show that $L(f)\cdot L(F)$ is self-adjoint.
 Let us check the first equation of~\eqref{eq:system}:
 \begin{multline}
\nonumber 
\tilde a_{ij,k}=(a_{is}
 A^s_j+f_i F_j+\bar{f}_i\bar{F}_j)_{,k}=
 a_{is,k}A^s_j+a^s_i
 A_{sj,k}+f_{i,k}F_j+f_i F_{j,k}+
 \bar{f}_{i,k}\bar{F}_j+\bar{f}_i\bar{F}_{j,k}
\\
  \stackrel{(\ref{eq:system})}{=}A^s_j f_i g_{sk}+A^s_j f_s g_{ik}+
 A^s_j \bar {f}_i J^{s'}_{\ \ s}g_{s' k}+A^s_j \bar{f}_s
 J^{\i}_{\ \ i}g_{\i k}+a^s_i F_j g_{sk}+a^s_i F_s g_{jk}+
 a^s_i \bar{F}_j J^{s'}_{\ \ s}g_{s' k}+a^s_i \bar{f}_s
 J^{\j}_{\ \ j}g_{\j k}\\
+\mu g_{ik} F_j-a_{ik}F_j+\Mu
 g_{jk}f_i-A_{jk}f_i+
 \mu J^{\i}_{\ \ i} g_{\i k} \bar{F}_j-J^{\i}_{\ \ i}
 a_{\i k}\bar{F}_j+\Mu J^{\j}_{\ \
 j}g_{\j k}\bar{f}_i-J^{\j}_{\ \ j}A_{\j k}\bar{f}_i\\
 =g_{ik}(f_s A^s_j+\mu F_j)+g_{jk}(F_s
 a^s_i+\Mu f_i)+
 J^{\i}_{\ \ i}g_{\i k}(\bar{f}_s A^s_j+\mu \bar{F}_j)+
 J^{\j}_{\ \ j}g_{\j k}(\bar{F}_s a^s_i+\Mu\bar{f}_i)
\\
  \stackrel{\eqref{eq:op_eq}}{=}\tilde{f}_i g_{j k} + \tilde{f}_j g_{i k}-J^{\alpha}_{i}\tilde{f}_\alpha J_{j k} - J^{\alpha}_{j}\tilde{f}_\alpha J_{i k}
 \end{multline}
 For the second equation one can calculate:
 \begin{multline}
\nonumber 
\tilde {f}_{i,k}=(\mu{F}_i+f_j A^j_i)_{,
 k}=\mu_{,k}F_i+\mu F_{i, k}+ f_{j,k}
 A^j_i+f^j A_{ij,k}\\
 \stackrel{(\ref{eq:system})}{=}-2f_k F_i+\mu\Mu g_{ik}-\mu A_{ik}+\mu A_{ik}
 -a_{jk}A^j_i+ f^j F_i g_{j k} + f^j F_j
 g_{i k}+f^j J^\j_{\ \ j}J^\i_{\ \ i}F_\i g_{\j k} +
 f^j J^\j_{\ \ j}J^\i_{\ \ i}F_\j g_{\i k}\\
 =(\mu \Mu+f^j F_j)g_{ik}-(f_k F_i
 +\bar{f}_k \bar{F}_i+A_{ij}
 a^j_k)+f^j\bar{F}_j J^\i_{\ \ i}g_{\i
 k}\stackrel{\eqref{eq:op_eq}}{=}
 \tilde \mu g_{ki}-\tilde a_{ki}
 \end{multline}
As we already stated in Remark \ref{rem:hpro2}, since $(\tilde{a}_{ij},\tilde{f}_{i})$ satisfies \eqref{eq:main}, the $1$-form $\tilde{f}_{i}$ is the differential of a function. It follows that $\tilde{f}_{i,j}$ is symmetric and from the last calculation we obtain that $\tilde a_{ij}$ is symmetric since it is the linear combination of two symmetric tensors. Let us now check the third equation
 of~\eqref{eq:system}:
 \begin{multline}
\nonumber 
\tilde \mu_{, i}=(\mu \Mu +f_k F^k)_{, i}=
 \mu_{, i} \Mu + \mu \Mu_{, i} +f_{k, i} F^k+f^k
 F_{k, i}\\
 \stackrel{(\ref{eq:system})}{=}-2f_i \Mu
 -2F_i \mu +F^k (\mu g_{ik}-a_{ik})+f^k (\Mu g_{ik}-A_{ik})=-(\mu F_i+f_k
 A^k_i)-(\Mu f_i+F_k
 a^k_i)\stackrel{\eqref{eq:op_eq}}{=}-2\tilde f_i
 \end{multline}
 Thus, $(\tilde{a}_{ij},\tilde{f}_{i},\tilde \mu)$ is a solution of~\eqref{eq:system}.\\
Now we show that the operator
 $L(F)=L^k(f)$ satisfies the conditions~\eqref{eq:op_eq}.
Since $L^{k}\cdot L=L\cdot L^{k}$, using~(\ref{eq:product}) we obtain
 $$\mu F_j+f_k A^k_j= \Mu f_j+a^k_j F_k$$
The last condition will be checked by induction. Suppose $f^i \bar {F}_i=0$ then
 $$f^i J^\i_{\ \ i}\widetilde{F}_\i=f^i \cdot J^\i_{\
 \ i}(\mu F_\i+f_k A^k_\i)=\mu\cdot 0 + f^k
 (J_{\ \ i}^\i A_{k \i})f^i=0.$$
From the constructions in \eqref{eq:def}, it is obvious that $L(\tilde{f})=L^{k+1}(f)$ where $\tilde{f}=-\frac{1}{2}\tilde{\mu}$. This completes the proof of the first part of Lemma \ref{lem:product}.\\
Part $(2)$ follows immediately from the already proven part: Proceeding by induction, we assume $L(f^{*k})=L^{k}(f)$. We already know from the first part of Lemma \ref{lem:product}, that $L^{k+1}(f)=L(f)\cdot L^{k}(f)=L(\tilde{f})$. The solution $\tilde{f}$ of \eqref{eq:tanno} is given by the formula $\tilde{f}=-\frac{1}{2}\tilde{\mu}=-2fF -\frac{1}{2}f_k F^k$, see the equations following formula \eqref{eq:product_sol}. By assumption $F=f^{*k}$ and therefore $\tilde{f}=f*f^{*k}=f^{*(k+1)}$.  
\end{proof}
Let us consider the natural operation of polynomials with real coefficients on elements of $T^{1}_{1}\widehat{M}$. From Lemma \ref{lem:product}, we immediately obtain that the family of $(1,1)$-tensors $L(f)$, constructed from solutions $f$ of \eqref{eq:tanno} is invariant with respect to this operation:
\begin{cor}
Let $f$ be a solution of \eqref{eq:tanno} with $c=1$ and $P(t)=c_k t^k+\dots+c_1 t+c_0$ an arbitrary polynomial with real coefficients. Then 
$P*(f):=c_k f^{*k}+c_{k-1} f^{*(k-1)}+\dots+c_1 f-\frac{1}{2}c_0$ is a solution of \eqref{eq:tanno} with $c=1$ and $L(P*(f))=P(L(f))$.
\end{cor}

\subsection{There exists a solution $f$ of \eqref{eq:tanno} such that $L(f)$ is a projector.}

We assume   that $(M, g, J)$ is a closed, connected Kähler manifold.  Our goal is to show   that 
the existence of a non-constant solution $f$ of (\ref{eq:tanno}) with $c=1$ implies the existence of a solution $\check{f}$ of 
(\ref{eq:tanno}) such that the matrix  $L(\check{f})$ is a non-trivial (i.e. $\neq 0$ and $\neq \Id$) projector. (Recall that a matrix $L$ is a \emph{projector}, if $L^2=L$.).
We need 
\begin{lem}
  \label{lem:minim}
  Let $(M, g, J)$ be a connected Kähler manifold and
  $f$  a solution of \eqref{eq:tanno} with $c=1$.
  \par Let $P(t)$ be the minimal polynomial of $L(f)$ at
  the point $\hat p \in\widehat M$. Then, $P(t)$ is the minimal
  polynomial of $L(f)$ at every point $\hat q\in \widehat M$.
\end{lem}
\begin{con}
We will always assume that the leading coefficient of a minimal polynomial is $1$.
\end{con} 
\begin{proof}
  As we have already proved, there exists a solution $\tilde{f}$ such that $$P(L(f))=L(\tilde{f}).$$ Since $P(L(f))$ vanishes at the
  point $\hat p=(x_+, x_-,p)$, we obtain that  $\tilde {a}_{ij}=0$, $\tilde{f}_{i}=0$
  and $\tilde \mu=0$ at $p$, where $(\tilde {a}_{ij},\tilde{f}_{i},\tilde \mu)$ denotes the solution of \eqref{eq:system} corresponding to $\tilde{f}$. Then, by
  Lemma~\ref{lem:frobenius}, the solution $(\tilde a_{ij},\tilde{f}_{i},\tilde
  \mu)$ of \eqref{eq:system} is identically zero on $M$. Thus, $P(L(f))$
  vanishes at all points of $\widehat M$.  It follows, that the polynomial $P(t)$
  is divisible by the minimal polynomial of $L(f)$ at $\hat
  q$. By the same reasoning (interchanging $p$ and $q$), we obtain
  that $Q(t)$ is divisible by $P(t)$.  Consequently, $P(t)=  Q(t)$.
\end{proof}

\begin{cor}
  The eigenvalues of $L(f)$ are constant functions on
  $\widehat M$.
\end{cor}
\begin{proof} By Lemma~\ref{lem:minim}, the minimal  polynomial 
does not depend on  the points of $\widehat M$. Then, the roots of the minimal  polynomial
  are also constant (i.e., do not depend on the points of  $\widehat M$). 
\end{proof}
In order to find the desired special solution of equation (\ref{eq:tanno}), we will use  that $M$ is closed.
\begin{lem}
  \label{lem:two_eigenvalues}
  Suppose $(M,g,J)$ is a closed, connected Kähler manifold. Let $f$  be a non-constant solution of \eqref{eq:tanno} with $c=1$.  Then at every point of $\widehat M$ the matrix $L(f)$ has at least two different real eigenvalues.
\end{lem}

\begin{proof}
  Let $(a_{ij},f_{i},\mu)$ denote the solution of \eqref{eq:system} corresponding to $f$. Since $M$ is closed, the function $\mu$ admits its maximal  and
  minimal values $\mu_{\max}$ and $\mu_{\min}$.  Let $p \in M$ be a 
  point where $\mu=\mu_{\max}$. At this point, $\mu_{,i}=0$ implying
  $f_i=\bar {f}_i=0$. Then, the matrix of $L(f)$ at $p$ has the form

  \begin{gather}
    L(f)=\left(
      \begin{array}{cc|ccc}
        \mu_{\max}&0&0&\dots&0\\
        0&\mu_{\max}&0&\dots&0\\
        \hline
        0&0&&&\\
        \vdots&\vdots&&a^i_j&\\
        0&0&&&
      \end{array}\right)
  \end{gather}

  Thus, $\mu_{\max}$ is an eigenvalue of $L(f)$ at $p$
  and, since the eigenvalues are constant, $\mu_{\mbox{\tiny max}}$ is an eigenvalue of $L(f)$ everywhere on $M$. The same
  holds for $\mu_{\min}$. Since $f_i\not\equiv 0$, $\mu$ is not
  constant implying $\mu_{\max}\ne \mu_{\min}$. Finally, $L(f)$ has two different real eigenvalues $\mu_{\max},
  \mu_{\min}$ at every point.
\end{proof}

\begin{rem}
  \label{-3}
  For further use let us note that 
   in the proof of Lemma~\ref{lem:two_eigenvalues} we have
  proved that if $\mu_{,i}=0$ at a point $p$ then $\mu(p)$ is an
  eigenvalue of $L$.
\end{rem}
Finally, let us show that there is always a solution of (\ref{eq:tanno}) of the desired special kind:
\begin{lem}
  Suppose $(M,g,J)$ is a closed and connected Kähler manifold. For every non-constant solution $f$ of \eqref{eq:tanno} with $c=1$, there exists a polynomial $P(t)$
  such that $P(L(f))$ is a non-trivial (i.e. it is neither
  ${\bf 0}$ nor $\Id$) projector.
\end{lem}

\begin{proof}
  We take a point $\hat{p}\in \widehat{M}$. By Lemma~\ref{lem:two_eigenvalues},
  $L(f)$ has at least two real eigenvalues at the
  point $\hat{p}$.  Then, by linear algebra, there exists a polynomial $P$
  such that $P( L(f))$ is a nontrivial projector
  at the point $\hat{p}$. Evidently, a matrix is a nontrivial projector,
  if and only if its minimal polynomial is $t(t-1)$ (multiplied by any
  nonzero constant). Since by Lemma~\ref{lem:minim} the minimal polynomial
  of $P( L(f))$ is the same at all points, the
  matrix $P( L(f))$ is a projector at every point
  of $\widehat{M}$.
\end{proof}

Thus, given a non-constant solution of \eqref{eq:tanno} with $c=1$ on a closed and connected Kähler manifold $M$, without loss of
generality we can think that a solution $f$ of (\ref{eq:tanno}) with $c=1$ is chosen such that the corresponding $(1,1)$-tensor $L(f)$ is a non-trivial
projector.

\subsection{Structure of eigenspaces of $a_j^i$, if $L(f)$ is a  nontrivial projector}

Assuming that $L(f)$ is  a nontrivial  projector, it has precisely two eigenvalues $1$ and $0$ and the $(2n+2)$-dimensional tangent space of
$\widehat M$ at every point $\hat x=(x_+,x_-,p)$ can be decomposed
into the sum of the corresponding eigenspaces $$T_{\hat x}\widehat M =
E_{L(f)}(1)\oplus E_{L(f)}(0).$$ 
The dimensions of $E_{L(f)}(1)$ and $E_{L(f)}(0)$ are even; we assume that the dimension  of $E_{L(f)}(1)$ is $2k+2$ and the dimension of $E_{L(f)}(0)$ is $2n-2k$. 
If $(a_{ij},f_{i},\mu)$ is the solution of \eqref{eq:system} corresponding to $f$, $\mu_{\max}$ and $\mu_{\min}$ are
eigenvalues of $L(f)$, see Lemma~\ref{lem:two_eigenvalues}. Since $L(f)$ is a projector we obtain $\mu_{\min}=0\leq\mu(x)\leq 1=\mu_{\max}$ on $M$. In view of Remark~\ref{-3}, the only critical values of $\mu$ are $1$ and $0$.
\begin{lem}
Let $f$ be a solution of (\ref{eq:tanno}) with $c=1$ such that $L(f)$ is a non-trivial projector. Let $(a_{ij},f_{i},\mu)$ be the solution of \eqref{eq:system} corresponding to $f$.   
Then, the following statements hold:
  \label{lem:eigen}
  \begin{enumerate}
  \item At the point $p$ such that $0<\mu< 1$,
    $a^i_j$ has the following structure of eigenvalues and eigenspaces
    \begin{enumerate}
    \item eigenvalue $1$ with geometric multiplicity $2k$;
    \item eigenvalue $0$ with geometric  multiplicity $(2n-2k-2)$;
    \item eigenvalue $(1-\mu)$ with multiplicity $2$.
    \end{enumerate}
  \item At the point $p$ such that $\mu=1$, $a^i_j$ has the following
    structure  of eigenvalues and eigenspaces:
    \begin{enumerate}
    \item eigenvalue $1$ with geometric multiplicity $2k$;
    \item eigenvalue $0$ with geometric  multiplicity $(2n-2k)$;
    \end{enumerate}
  \item At the point $p$ such that $\mu=0$, $a^i_j$ has the following
    structure of   eigenvalues and eigenspaces:
    \begin{enumerate}
    \item eigenvalue $1$ with geometric multiplicity $2k+2$;
    \item eigenvalue $0$ with geometric multiplicity $(2n-2k-2)$.
    \end{enumerate}
  \end{enumerate}
\weg{  (see Fig.~\ref{fig:eigenspaces}).}
\end{lem}

\begin{con} We identify  $M$  with the set $(0,0)\times M\subset \widehat M$. 
This identification allows us to consider   $T_xM$ as  a  linear subspace of $T_{(0,0)\times x} \widehat M$: the  vector $(v_1,...,v_n)\in T_xM$ is identified with $(0,0,v_1,...,v_n)\in T_{(0,0)\times x} \widehat M$. 
\end{con} 

\begin{proof}
  Consider the point $p$ such that $0<\mu< 1$. For any vector $v\in E_{L(f)}(1)\cap T_{p}M$ we  calculate
  \begin{gather}
    L(f) v=\left(
      \begin{array}{cc|ccc}
        \mu&0&f_{1}&\dots&f_{2n}\\
        0&\mu&\bar{f}_{1}&\dots&\bar{f}_{2n}\\
        \hline
        f^{1}&\bar{f}^{1}&&&\\
        \vdots&\vdots&&a^i_j&\\
        f^{2n}&\bar{f}^{2n}&&&
      \end{array}\right)
    \left(
      \begin{array}{c}
        0\\
        0\\
        v^1\\
        \vdots\\
        v^{2n}
      \end{array}
    \right)=
    \left(
      \begin{array}{c}
        f_j v^j\\
        \bar {f}_j v^j\\
        \\
        a^i_j v^j\\
        \\
      \end{array}
    \right)
    = \left(
      \begin{array}{c}
        0\\
        0\\
        v^1\\
        \vdots\\
        v^{2n}
      \end{array}
    \right)
  \end{gather}

  Thus, $v=(v^1,\dots v^{2n})$ is an eigenvector of $a^i_j$ with
  eigenvalue $1$, hence $E_{L(f)}(1)\cap T_{p}M$ is contained in the eigenspace $E_{1}$ of $a^{i}_{j}$ with eigenvalue $1$. Similarly, any $v \in E_{L(f)}(0)\cap T_{p}M$ is an eigenvector of $a^i_j$ with eigenvalue~$0$ and $E_{L(f)}(0)\cap T_{p}M$ is contained in the eigenspace $E_{0}$ of $a^{i}_{j}$ with eigenvalue $0$. Note that the dimension of $E_{L(f)}(1)\cap T_{p}M$ is at least  
  $2k$, and the dimension of $E_{L(f)}(0)\cap T_{p}M$ is at least $2n - 2k - 2$. Let us now show that $f^i$ and $\bar{f}^i$ are eigenvectors of $a^i_j$ with eigenvalue $(1-\mu)$. We multiply the first basis vector $(1,0,\dots,0)$ by the matrix $L(f)^2-L(f)$ (which is identically zero) and obtain 
  \begin{gather}
    0 = (L(f)^2-L(f)) \left(
      \begin{array}{c}
        1\\
        0\\
        0\\
        \vdots\\
        0
      \end{array}
    \right)=
    \left(
      \begin{array}{c}
        \mu^2 + f_j f^j - \mu\\
        \bar{f}_i f^i\\
        \\
        \mu f^i + a^i_j f^j-f^i\\
        \\
      \end{array}
    \right)
  \end{gather}
  This gives us the necessary equation $a^i_j f^j= (1-\mu)f^i$. Considering the same procedure for the second basis vector $(0,1,\dots,0)$ we obtain that $\bar{f}^{i}$ is an eigenvector of $a^{i}_{j}$ to the eigenvalue $(1-\mu)$ and hence, the dimension of the eigenspace $E_{1-\mu}$ of $a^{i}_{j}$ corresponding to the eigenvalue $1-\mu$ is at least $2$.    
On the one hand, $dim E_{1}+dim E_{0}+dim E_{1-\mu}$ is at most $2n$ but on the other hand, $2k \leq dim(E_{L(f)}(1)\cap T_{p}M)\leq dim E_{1}$, $2n - 2k - 2\leq dim(E_{L(f)}(0)\cap T_{p}M)\leq dim E_{0}$ and $2 \leq dim E_{1-\mu}$ implying that 
\begin{align}
T_{p}M=E_{1}\oplus E_{0}\oplus E_{1-\mu}. 
\end{align}
Furthermore $E_{1}=E_{L(f)}(1)\cap T_{p}M$, $E_{0}=E_{L(f)}(0)\cap T_{p}M$ and $E_{1-\mu}=\spann\{f^{i},\bar{f}^{i}\}$ are of dimensions $2k$, $2n-2k-2$ and $2$ respectively.
 The proof at the points such that  $\mu=0$ or  $\mu=1$ is similar.
\end{proof}

\subsection{ If there exists a solution $f$ of equation (\ref{eq:tanno}) such that $L(f)$ is a non-trivial projector,
 the metric $g$ is positively definite on a closed $M$}
Above we have proven that, given a non-constant solution of \eqref{eq:tanno} with $c=1$, there always exists a solution $f$ of (\ref{eq:tanno})
such that the corresponding matrix $L(f)$ is a non-trivial projector. If $(a_{ij},f_{i},\mu)$ is the solution of \eqref{eq:system} corresponding to $f$, this implies that the eigenvalues and the dimension of the eigenspaces of $a^i_j$ are given by Lemma \ref{lem:eigen}. Now we are ready to prove that $g$ is positively definite
(as we claimed in Theorem~\ref{pos}).

Let us consider such a solution $f$ and the corresponding solution $(a_{ij},f_{i},\mu)$ of \eqref{eq:system}. We rewrite the second equation in (\ref{eq:system}) in the form
\begin{gather}
  \label{eq:mu}
  \mu_{,i j} =2 a_{i j} -2 \mu\, g_{i j}
\end{gather}
Let $p$ be a point where $\mu$ takes its maximum value $1$. As we have already shown
$f^i(p)=0$ and the tangent space $T_pM$ is equal to the direct sum of
the eigenspaces:
$$T_p M=E_1\oplus E_0$$
Consider the restriction of~(\ref{eq:mu}) to $E_0$. Since the restriction of
the bilinear form $a_{i j}$ to $E_0$ is identically zero, the equation
takes the form:
$$\left.\mu_{,i j}\right|_{E_0}=-2 \left.g_{i j}\right|_{E_0}.$$
But $\mu_{,i j}$ is the Hessian of $\mu$ at the maximum point
$p$. Then, it is non-positively definite. Hence, the non-degenerate metric
tensor $g_{ij}$ is positively definite on $E_0$ at $p$.
Let us now consider the distribution of the orthogonal complement
$E_1^\bot$, which is well-defined, smooth and integrable in
a neighborhood of $p$. The restriction of the metric $g$ onto $E_1^\bot$ is
non-degenerate and is positively definite in one point. Hence, by
continuity it is positively definite in a whole neighborhood.
Similarly, at a minimum point $q$ one can consider the restriction
of~(\ref{eq:mu}) to $E_1$:
$$\left.\mu_{,i j}\right|_{E_1}=2 \left.a_{i j}\right|_{E_1}=2
\left.g_{i j}\right|_{E_1},$$
since $a^{i}_{j\,|E_{1}}=\delta^{i}_{j\,|E_{1}}$, which implies that $g$ is positively definite on $E_1$ at $q$.
Considering the distribution $E_0^\bot$, we obtain that the restriction of $g$
to $E_0^\bot$ is positively definite in a neighborhood of $q$. 
Let us now consider the general point $x$, where
$$T_xM=E_1\oplus E_0 \oplus \spann\{f^i,\bar{f}^i\}.$$
We choose a piecewise smooth path $\gamma:[0,1]\rightarrow M$, connecting a point $p=\gamma(0)$ where $\mu(p)=1$ with the point $x=\gamma(1)$, such that there is no point along $\gamma$ where $\mu=0$. Since the distribution $E_{1}^{\perp}$ is differentiable on $M\setminus\{q\in M:\mu(q)=0\}$, there cannot be a change of signature of the restriction of the metric $g_{|E_{1}^{\perp}}$ along that path, unless the determinant of $g_{|E_{1}^{\perp}}$ vanishes at some point between $p$ and $x$. Since $g$ is nondegenerate, this never can happen and we obtain that $g_{|E_{1}^{\perp}}$ is positively definite at $x$. Exactly the same arguments can be used if one wants to show that $g_{|E_{0}^{\perp}}$ is positively definite at $x$. In the end, we obtain that $g$ is positively definite at each point of $M$ and hence, Theorem~\ref{pos} is proven.

\section{Proof of Theorem \ref{thm:main}}
\label{sec:proof}
Let $(M,g,J)$ be a closed, connected pseudo-Riemannian Kähler manifold and $f$ a non-constant solution of equation \eqref{eq:tanno}.
By Theorem \ref{thm:c=0}, $c\neq 0$ and the metric $g$ can be replaced by $\tilde{g}=c\cdot g$ without changing the Levi-Civita connection. The function $f$ is now a non-constant solution of \eqref{eq:tanno} with $c=1$ and using Theorem \ref{pos} we obtain that $\tilde{g}$ has to be positive-definite. Applying the classical result of Tanno \cite{Tanno1978} for positive-definite metrics, we obtain that $(M,4cg,J)$ has constant holomorphic sectional curvature equal to $1$, hence Theorem \ref{thm:main} is proven.

\nocite{*}
\bibliographystyle{plain}

\end{document}